\documentclass[onefignum,onetabnum]{siamart190516}

\usepackage{lipsum}
\usepackage{amsfonts}
\usepackage{graphicx}
\usepackage{epstopdf}

\usepackage{epic}
\usepackage{eepic}
\usepackage{tikz}

\usepackage{algorithmic}
\ifpdf
  \DeclareGraphicsExtensions{.eps,.pdf,.png,.jpg}
\else
  \DeclareGraphicsExtensions{.eps}
\fi

\newsiamremark{remark}{Remark}
\newsiamremark{hypothesis}{Hypothesis}
\crefname{hypothesis}{Hypothesis}{Hypotheses}
\newsiamthm{claim}{Claim}

\newcommand{\nmbr}[1]{\mathbb{#1}}
\newcommand{\bc}[1]{{\mathcal #1}}
\newcommand{\alg}[1]{{\mathbf #1}}

\graphicspath{{./figs/}}

\headers{Generalized spectrum of second order differential operators}{T. Gergelits, B. F. Nielsen, Z. Strako{\v s}}

\title{Generalized spectrum of second order differential operators
\thanks{Submitted to the editors January 31st, 2020.
\funding{The work of Tom{\'a}{\v s} Gergelits and Zden{\v e}k Strako{\v s} has been supported by the Grant Agency of the Czech Republic under the contract No. 17-04150J, and the work of Bj{\o}rn F. Nielsen was supported by The Research Council of Norway, project number 239070.}
}
}

\author{
Tom{\'a}{\v s} Gergelits\footnotemark[2]\ \footnotemark[3]
\and Bj{\O}rn Fredrik Nielsen\footnotemark[4]
\and Zden{\v e}k Strako{\v s}\footnotemark[3]
}

\usepackage{amsopn}

\ifpdf
\hypersetup{
  pdftitle={Generalized spectrum of second order differential operators},
  pdfauthor={T. Gergelits, B. F. Nielsen, Z. Strako{\v s}}
}
\fi

\begin{document}

\maketitle

\renewcommand{\thefootnote}{\fnsymbol{footnote}}
\footnotetext[2]{Institute of Computer Sciences, Czech Academy of Science, Prague, Czech Republic (\email{gergelits@cs.cas.cz}).}
\footnotetext[3]{Faculty of Mathematics and Physics, Charles University, Prague, Czech Republic (\email{gergelits@karlin.mff.cuni.cz}, \email{strakos@karlin.mff.cuni.cz}).}
\footnotetext[4]{Faculty of Science and Technology, Norwegian University of Life Sciences, NO-1432 {\AA}s, Norway
  (\email{bjorn.f.nielsen@nmbu.no}).}

\renewcommand{\thefootnote}{\arabic{footnote}}

\begin{abstract}
We analyze the spectrum of the operator $\Delta^{-1} [\nabla \cdot (K\nabla u)]$, where $\Delta$ denotes the Laplacian and $K=K(x,y)$ is a symmetric tensor. Our main result shows that this spectrum can be derived from the spectral decomposition $K=Q \Lambda Q^T$, where $Q=Q(x,y)$ is an orthogonal matrix and $\Lambda=\Lambda(x,y)$ is a diagonal matrix. More precisely, provided that $K$ is continuous, the spectrum equals the convex hull of the ranges of the diagonal function entries of $\Lambda$. The involved domain is assumed to be bounded and Lipschitz, and both homogeneous Dirichlet and homogeneous Neumann boundary conditions are considered. We study operators defined on infinite dimensional Sobolev spaces. Our theoretical investigations are illuminated by numerical experiments, using discretized problems. 

The results presented in this paper extend previous analyses which have addressed elliptic differential operators with scalar coefficient functions. Our investigation is motivated by both preconditioning issues (efficient numerical computations) and the need to further develop the spectral theory of second order PDEs (core analysis). 
\end{abstract}

\begin{keywords}
  Second order PDEs, generalized eigenvalues, spectrum, tensors, preconditioning.
\end{keywords}

\begin{AMS}
  65F08, 65F15, 65N12, 35J99
\end{AMS}

\section{Introduction} 
\label{introduction}
For simple domains, the eigenfunctions and eigenvalues of the Laplacian $\Delta$ can be characterized in terms of trigonometric functions. Similar analytic information about the spectrum of general second order differential operators $\nabla \cdot (K \nabla u)$ is not available. On the other hand, in \cite{Ger19,Nie09} the authors show that the generalized eigenvalue problem 
\[
\nabla \cdot (k \nabla u) = \lambda \Delta u \quad \mbox{ for } (x,y) \in \Omega, 
\]
where $k$ is a uniformly positive scalar function, can be analyzed in detail.
More specifically, if $k$ is continuous, then the range 
\[k(\Omega) = \left\{k(x,y),\, (x,y) \in \Omega\right\}\]
of $k$ is contained in the spectrum of the  operator $\Delta^{-1} [\nabla \cdot (k\nabla u)]$. Furthermore, for discretized problems,  assuming that $k$ is bounded and piecewise continuous, the function  values of $k$ over the patches defined by the discretization basis functions provide accurate approximations of the generalized eigenvalues. 

The main purpose of this paper is to extend the results published in \cite{Ger19,Nie09} to second order differential operators which involve a symmetric tensor.  That is, to the generalized eigenvalue problem
\begin{equation}
\begin{split}
\label{eq:eigen}
\nabla \cdot (K \nabla u) &= \lambda \Delta u \quad \mbox{ for } (x,y) \in \Omega, \\
u &= 0 \quad\mbox{ for } (x,y) \in\partial\Omega,
\end{split}
\end{equation}
where the open domain $\Omega \subset \nmbr{R}^2$ is bounded and Lipschitz, and the real valued tensor function $K: \Omega \to \nmbr{R}^{2\times 2}$ is symmetric with its entries being bounded Lebesgue integrable functions and
with the spectral decomposition
\begin{equation}\label{eq:decomposition} 
\begin{aligned}
K(x,y)&= Q(x,y) \Lambda (x,y)  Q^T(x,y), \quad(x,y)\in\Omega,\\
\Lambda (x,y) &= \left[ 
\begin{array}{cc}
\kappa_1(x,y) & 0 \\
0 & \kappa_2(x,y)
\end{array}
\right] , \quad
Q Q^T = Q^T Q = I.
\end{aligned} 
\end{equation}
More specifically, defining the operators $\mathcal{L}, \, \mathcal{A}: H_0^1(\Omega) \mapsto H^{-1}(\Omega)$ as  
\begin{align}\label{eq:L}
& \langle \mathcal{L} \phi, \psi \rangle = \int_{\Omega} \nabla \phi  \cdot  \nabla \psi,  \quad \phi,  \psi \in H_0^1(\Omega), \\
\label{eq:A}
& \langle \mathcal{A} \phi, \psi \rangle = \int_{\Omega} K \nabla \phi  \cdot  \nabla \psi, \quad \phi,  \psi \in H_0^1(\Omega),
\end{align} 
we characterize the spectrum of the preconditioned operator
\begin{equation}\label{eq:operator}
\bc{L}^{-1}\bc{A}: H_0^1(\Omega)  \to H_0^1(\Omega), 
\end{equation}
defined as
\footnote{For operators defined on infinite dimensional Hilbert (Sobolev) spaces, the eigenvalues represent, in general, only a part of the spectrum. Therefore, the generalized eigenvalue problem \cref{eq:eigen} does not determine the whole spectrum \cref{eq:spectrum}.} 
\begin{equation}\label{eq:spectrum}
\mathrm{sp}(\mathcal{L}^{-1} \mathcal{A}) 
\equiv \left\{ \lambda \in \mathbb{C}; \,  \lambda \mathcal{I} - \mathcal{L}^{-1} \mathcal{A} \mbox{ does not have a bounded inverse} \right\}.
\end{equation}

This  paper proves the following  result:
\begin{theorem}[Spectrum of the preconditioned operator]
\label{th:theorem}
Consider an open and bounded Lipschitz domain $\Omega\subset \nmbr{R}^2$.
Assume that the tensor $K$ is symmetric and continuous throughout the closure $\overline{\Omega}$.
Then the spectrum of the operator $\bc{L}^{-1}\bc{A}$, defined in \cref{eq:L}--\cref{eq:spectrum}, equals 
\begin{equation}\label{eq:result}
 \mathrm{sp}(\mathcal{L}^{-1} \mathcal{A})) = \mathrm{Conv}(\kappa_1(\overline{\Omega}) \cup \kappa_2(\overline{\Omega})),
\end{equation}
where
\begin{equation}\label{eq:conv:hull}
 \mathrm{Conv}(\kappa_1(\overline{\Omega}) \cup \kappa_2(\overline{\Omega})) = [\inf_{(x,y)\in\overline{\Omega}}\min_{i=1,2}\ \kappa_i(x,y),\,\sup_{(x,y)\in\overline{\Omega}}\max_{i=1,2}\ \kappa_i(x,y)].
\end{equation}

\end{theorem}

Note that this theorem extends the results in \cite{Nie09} in several ways. It holds for second order differential operators with definite, indefinite and semidefinite tensors. 
Moreover, instead of the inclusion proved for the scalar case in \cite{Nie09}, it shows that the spectrum actually equals the interval \cref{eq:conv:hull} determined by  $K(x,y)$.

Our theoretical study addresses operators defined on infinite dimensional Sobolev spaces. Numerical experiments suggest that even stronger properties, analogous to the scalar case analyzed in \cite{Ger19}, hold for discretized problems. 


Our theoretical results can be illustrated by the following experiment.
We consider three test problems \cref{eq:eigen}  with diagonal tensors \cref{eq:decomposition} (i.e., $Q=I$) defined on the domain $\Omega\equiv(0,1)\times(0,1)$, where 
\begin{equation}\label{eq:intro:exp}
\begin{aligned}
(P1):\quad \kappa_1(x,y) &= 1,\qquad\qquad\qquad\quad \kappa_2(x,y) = 10,\\
(P2):\quad \kappa_1(x,y) &= 1+0.5(x+y),\ \quad \kappa_2(x,y) = 10-0.5(x+y), \\
(P3):\quad \kappa_1(x,y) &= 1+3(x+y),\quad \quad \kappa_2(x,y) = 10-2(x+y),
\end{aligned}
\end{equation}
for $(x,y)\in\Omega$. We discretize the problem \cref{eq:eigen} using a uniform triangular mesh with piecewise linear discretization basis functions; see \cite{Ger19} for the scalar case analogy. \Cref{fig:intro} presents the eigenvalues of the resulting discrete generalized eigenvalue problem of size $381$. We observe that the spectrum of the discretized problem covers not only the union of the ranges $\kappa_1(\overline{\Omega})\cup\kappa_2(\overline{\Omega})$, but in the case that $\kappa_1(\overline{\Omega})$ and $\kappa_2(\overline{\Omega})$ do not overlap, it surprisingly covers the whole interval \cref{eq:conv:hull}.
\begin{figure}[ht]
\centering
\includegraphics[width = 0.7\linewidth]{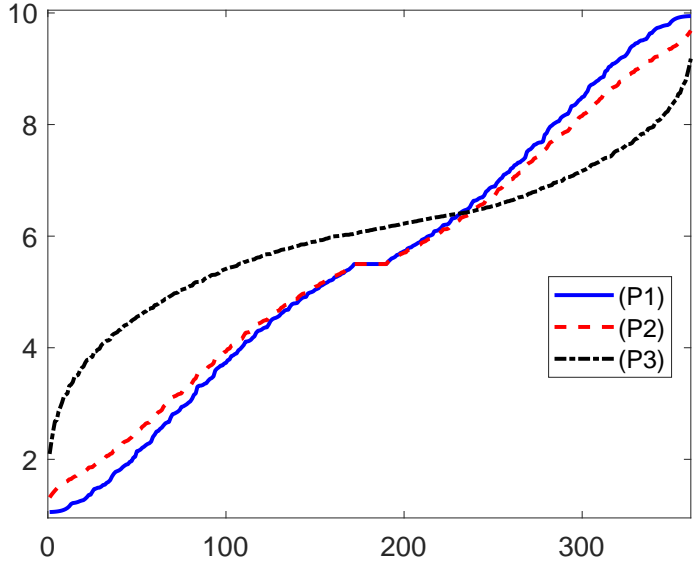}
\caption{Eigenvalues of the  discretized problems (P1)--(P3), defined in \cref{eq:intro:exp}, spread over the entire interval $[1,10]$, while the ranges of entries of the diagonal tensor are the following:
(P1): $\kappa_1(\overline{\Omega}) = 1,\ 		\kappa_2(\overline{\Omega}) = 10$;
(P2): $\kappa_1(\overline{\Omega}) = [1,2],\  \kappa_2(\overline{\Omega}) = [9,10]$;
(P3): $\kappa_1(\overline{\Omega}) = [1,7],\  \kappa_2(\overline{\Omega}) = [6,10]$. Horizontal axis: the indices of the increasingly ordered eigenvalues. Vertical axis: the size of the eigenvalues.}
\label{fig:intro}
\end{figure}

Since $\langle \mathcal{A} u, v \rangle =  \langle \mathcal{A} v, u \rangle$ for all $u, \, v \in H_0^1(\Omega)$, which is a consequence of the symmetry of the tensor $K$, the preconditioned operator \cref{eq:operator} is self-adjoint with respect to the inner product associated with the Laplacian:
\begin{align}\label{eq:inner}
(u,v)_{\bc{L}} &\equiv \langle \mathcal{L} u, v \rangle =  \int_{\Omega} \nabla u  \cdot  \nabla v, \quad u,v\in H_0^1(\Omega), \\
 \label{eq:self:L}
(\bc{L}^{-1}\bc{A}u,v)_{\bc{L}} &= \langle \mathcal{A} u, v \rangle =  \langle \mathcal{A} v, u \rangle = (\bc{L}^{-1}\bc{A}v,u)_{\bc{L}}.
\end{align}
Consequently, $\mathrm{sp}(\bc{L}^{-1}\bc{A})\subset \nmbr{R}$.
The inner product \cref{eq:inner} defines the norm 
\[\|u\|^2_{\bc{L}}\equiv (u,u)_{\bc{L}} = \langle \mathcal{L} u, u \rangle =  \int_{\Omega} \|\nabla u  \|^2 = \int_{\Omega} \| u_x \|_2^2 + \| u_y \|_2^2, \quad u\in H_0^1(\Omega) \]
used in the proofs below.

The convergence behavior of the Conjugate Gradient (CG) method is determined by the spectral distribution functions of the involved linear systems; see, e.g.,  \cite{Gre89,LSB13}. Hence, the analysis presented in this paper can be employed to better understand the performance of CG when the inverse of the Laplacian (or some variant incorporating it) is applied as preconditioner to solve discretized second order elliptic PDEs; see \cite{Ger19} for a discussion of this topic. Also, constant-coefficient-preconditioners may be of particular interest when the Isogeometric Analysis (IgA) approach is employed to discretize both PDEs and the involved computational domains in terms of B-splines
\cite{Hof17,Hof15,San16}.  

This paper is organized in the following way. For clarity of exposition,  we restrict ourselves in \cref{sec:diag,sec:diag:ii} to problems with diagonal tensors. In \cref{sec:diag}  we present auxiliary lemmas generalizing, step by step, the results in \cite{Nie09}. \Cref{sec:diag:ii} contains the proof of the main result for problems with diagonal tensors, and in \cref{sec:symm} we generalize the lemmas from previous sections to nondiagonal symmetric tensors and give the proof of the main result \cref{th:theorem}. In \cref{sec:neumann} we comment on problems with homogenous Neumann boundary conditions. The numerical experiments  in \cref{sec:exp} illustrate the results of the analysis, and the text closes with 
a brief discussion of some open problems in \cref{sec:remarks}.


\section{Auxiliary results}\label{sec:diag}
We will start with considering diagonal tensors, i.e.,
\begin{equation}\label{eq:diag:tensor}
K(x,y)=\left[ 
\begin{array}{cc}
\kappa_1(x,y) & 0 \\
0 & \kappa_2(x,y)
\end{array}
\right].
\end{equation}
This will allow us to explain with full clarity the main difference between the scalar case studied in \cite{Ger19,Nie09} and the tensor case analyzed in this paper.

\subsection{Function values at points of continuity belong to the spectrum}
The following lemma generalizes statement (a) in Theorem~3.1 in \cite{Nie09}.
\begin{lemma}\label{th:function:values}
Assume that $K$ is a diagonal tensor, where the entries $\kappa_1$ and $\kappa_2$ are bounded and Lebesgue integrable functions on $\Omega$. 
The following holds for $i=1, \,2$: 
If $\kappa_i$ is continuous at $(x_0,y_0) \in \Omega$, then  $$\kappa_i(x_0,y_0) \in \mathrm{sp}(\mathcal{L}^{-1} \mathcal{A}).$$
\end{lemma}
\begin{proof}
Assume that $\kappa_1$ is continuous at $(x_0,y_0)$ and let $$\lambda\equiv\kappa_1(x_0,y_0).$$
We will construct parametrized functions $v_r$ and $u_r = (\lambda \bc{I} - \bc{L}^{-1}\bc{A})v_r$ such that 
\begin{equation}
\label{BFnew1}
\lim_{r\to0}\|v_r\|_{\bc{L}}\neq 0\quad \mbox{and}\quad \lim_{r\to0}\|u_r\|_{\bc{L}}=0,
\end{equation}
which is not possible if $\lambda \bc{I} - \bc{L}^{-1}\bc{A}$ has a bounded inverse: $v_r = (\lambda \bc{I} - \bc{L}^{-1}\bc{A})^{-1}u_r$ and $\lim_{r\to0}\|u_r\|_{\bc{L}}=0$ imply that $\lim_{r\to0}\|v_r\|_{\bc{L}}=0$. (The norm $\|\cdot\|_{\bc{L}}$ is the norm induced by the inner product \cref{eq:inner}).

The functions $v_r$ can be constructed, e.g., in the following way. Consider, for a sufficiently small $r>0$, the following closed neighborhood of the point $(x_0,y_0)$:
\begin{equation}\label{eq:domain:r}
R_r = [x_0-r^2,x_0+r^2] \times [y_0-r,y_0+r]\subset \Omega. 
\end{equation}
For $(x,y)\in R_r$ define
\begin{equation}
\label{B0}
v_r(x,y) = 
\sqrt{r} \min \left\{ 1-\tfrac{|x-x_0|}{r^2},  \, \tfrac{1}{r} -\tfrac{|y-y_0|}{r^2}  \right\}, 
\end{equation}
and $v_r(x,y) = 0$ otherwise.
It can be verified that {(see \cref{sec:apA})}
\begin{equation}\label{B1}
\begin{aligned}
4-4r\leq \| (v_r)_x \|_{L^2(\Omega)}^2 &\leq 4, \\
\| (v_r)_y \|_{L^2(\Omega)}^2 &\leq 4r.
\end{aligned}
\end{equation}
Consequently,
\begin{equation}
\label{B2}
\lim_{r \rightarrow 0} {\| v_r \|_{\bc{L}} = \lim_{r \rightarrow 0} \left(\|(v_r)_x\|^2_{L^2(\Omega)}+\|(v_r)_y\|^2_{L^2(\Omega)}\right)^{1/2} }= 2.
\end{equation}
Considering
\begin{equation}
\label{B3}
u_r = (\lambda \mathcal{I} - \mathcal{L}^{-1} \mathcal{A}) v_r, \quad\mbox{i.e.,}\quad  \bc{L} u_r = ({\lambda} \bc{L} - \bc{A}) v_r, 
\end{equation}
we get
\begin{align*}
\| u_r \|^2_{\bc{L}} = \langle \mathcal{L} u_r, u_r \rangle
	&= \langle (\lambda \mathcal{L} - \mathcal{A}) v_r  , u_r \rangle \\
	&= \int_{\Omega} (\lambda I - K) \nabla v_r  \cdot  \nabla u_r \\
	& \leq \left( \int_{\Omega} |(\lambda I - K) \nabla v_r|^2 \right)^{1/2} \| u_r \|_{\bc{L}}. 
\end{align*}
Using that $\mathrm{supp} (v_r) = R_r$ and \cref{B1},
\begin{align*}
\nonumber
\|  u_r \|^2_{\bc{L}} &\leq \| (\lambda-\kappa_1) (v_r)_x \|_{L^2(\Omega)}^2+\| (\lambda-\kappa_2) (v_r)_y \|_{L^2(\Omega)}^2 \\
& \leq 4\sup_{(x,y) \in R_r} |\kappa_1(x_0,y_0)-\kappa_1(x,y) |^2  + 4r(\| \kappa_1 \|_{L^{\infty}(\Omega)}+\| \kappa_2 \|_{L^{\infty}(\Omega)})^2,
\end{align*}
and from the continuity of $\kappa_1(x,y)$ at $(x_0,y_0)$, 
\begin{equation}
\label{B4}
\lim_{r \rightarrow 0} \| u_r \|_{\bc{L}} = 0. 
\end{equation}

From \cref{B2} and \cref{B4} we conclude that we can construct functions $v_r$ and $u_r= (\lambda \bc{I} - \bc{L}^{-1}\bc{A})v_r$ such that \cref{BFnew1} holds. We conclude that $\kappa_1(x_0,y_0) \bc{I} - \bc{L}^{-1}\bc{A}$ can not have a bounded inverse.

The proof that $\kappa_2(x_0,y_0)$ belongs to the spectrum if $\kappa_2$ is continuous at $(x_0,y_0)$ is trivially analogous.
\end{proof}

If $\kappa_i\in \bc{C}(\Omega)$, $i = 1,2$, then \cref{th:function:values} gives a diagonal-tensor-case analogy of Theorem 3.1, statement (b), in \cite{Nie09}. 
As is shown next, in the tensor case the spectrum of the preconditioned operator $\bc{L}^{-1}\bc{A}$ can, however, also contain numbers that do not belong to any of the individual ranges of the functions $\kappa_1$ and $\kappa_2$.

\subsection{Disjoint ranges extend the spectrum}
An unexpected case occurs when the ranges of $\kappa_1$ and $\kappa_2$ are disjoint,
\[\kappa_1(\overline{\Omega}) \cap \kappa_2(\overline{\Omega}) = \emptyset.\]
We begin by presenting the following facts that will be used in the proofs. 

\subsubsection{Dirichlet problem for the wave equation}\label{sec:dir}
Note that for any integer $n$,
\begin{equation}
\label{B5.01}
\phi(x,y) = \sin(n \pi c l^{-1} (y-y_0)) \sin(n \pi  l^{-1} (x-x_0))
\end{equation}
solves the following Dirichlet problem for the wave equation:  
\begin{equation}
\label{B5.1}
\begin{split}
\phi_{yy} &= c^2 \phi_{xx} \quad \mbox{in } \Sigma_l, \\
\phi & = 0 \quad \mbox{on } \partial \Sigma_l, 
\end{split}
\end{equation}
where $l$ is a positive constant which determines the size of the solution domain 
\[
\Sigma_l = (x_0,x_0 + l) \times (y_0,y_0+l/c),
\]
and $c>0$ is arbitrary.
We conclude that this Dirichlet problem has infinitely many nontrivial solutions. 
It is also clear that  $\Sigma_l$ can be made as small as needed by choosing $l>0$ sufficiently small. 

\subsubsection{Tensors  constant on an open subdomain} 
Consider the generalized eigenvalue problem \cref{eq:eigen}
with a diagonal tensor $K(x,y)$ \cref{eq:diag:tensor} that is constant on an open subdomain $S \subset \Omega$. Then we get the following lemma.

\begin{lemma}
\label{th:constant}
Consider a diagonal tensor \cref{eq:diag:tensor}, where the bounded and Lebesgue integrable functions $\kappa_i$, $i = 1,2$  are constant on an open subdomain $S\subset\Omega$. Assuming that
\begin{equation}
\label{B6.1}
\sup_{(x,y)\in\Omega} \kappa_1(x,y) < \inf_{(x,y)\in\Omega}\kappa_2(x,y),
\end{equation}
then the following closed interval belongs to the spectrum of $\bc{L}^{-1}\bc{A}$,
\begin{equation}\label{eq:constant}
[\sup_{(x,y)\in\Omega} \kappa_1(x,y), \, \inf_{(x,y)\in\Omega} \kappa_2(x,y)]  \subset \mathrm{sp}(\mathcal{L}^{-1} \mathcal{A}).
\end{equation}
The analogous statement obviously holds with interchanging the roles of $\kappa_1$ and $\kappa_2$ in \cref{B6.1} and \cref{eq:constant}.
\end{lemma}

\begin{proof}
Consider an arbitrary fixed point $(x_0,y_0)\in S$. For any fixed  $c>0$ there is  $l > 0$ such that 
\[
\Sigma_l  \equiv (x_0,x_0 + l) \times (y_0,y_0+l/c) \subset S. 
\]
Since $K(x,y)$ is constant on $\Sigma_l$, we can rewrite $\cref{eq:eigen}$ as 
\begin{equation}\label{eq:subregion}
(\lambda - \overline{k}_1) v_{xx} + (\lambda - \overline{k}_2) v_{yy} = 0  \quad \mbox{in } \Sigma_l,  
\end{equation}
where $\overline{k}_1$ and $\overline{k}_2$ are constants and
\[
K(x,y) = \left[ 
\begin{array}{c c}
\overline{k}_1 & 0 \\
0 & \overline{k}_2
\end{array} 
\right], \quad (x,y) \in \Sigma_l. 
\] 

Consider an arbitrary $\lambda$ in the interval $(\overline{k}_1,\overline{k}_2)$. Then \cref{eq:subregion} represents, with
\begin{equation}
\nonumber
c^2 = \frac{\lambda - \overline{k}_1}{\overline{k}_2 - \lambda} > 0, 
\end{equation}
the wave equation \cref{B5.1}. 
Taking any nontrivial solution $\phi$ of \cref{B5.1}, the function $v$  defined on $\Omega$ as
\[
v(x,y)= \left\{ 
\begin{array}{ll}
\phi(x,y), & \quad (x,y)\in\Sigma_l, \\
0, & \quad (x,y)\notin \Sigma_l,
\end{array}
\right.
\]
solves the weak form of the generalized eigenvalue problem \cref{eq:eigen}. 
We conclude that $(\overline{k}_1,\overline{k}_2) \subset \mathrm{sp}(\mathcal{L}^{-1} \mathcal{A})$. 

Since, by construction,
\begin{align}\label{eq:con:proof}
\overline{k}_1 \leq \sup_{(x,y)\in\Omega} \kappa_1(x,y)< \inf_{(x,y)\in\Omega}\kappa_2(x,y)) \leq \overline{k}_2,
\end{align}
it remains to prove that, if the equality is attained at any side of \cref{eq:con:proof}, then the associated $\overline{k}_i$, $i=1$ and/or $i=2$, also belongs to the spectrum of $\bc{L}^{-1}\bc{A}$. But this is trivially true using \cref{th:function:values} because $\overline{k}_i$ is a function value of $\kappa_i(x,y)$ at $\Sigma_l$ where $\kappa_i$ is constant and therefore continuous.
\end{proof}

\Cref{th:constant} shows that, under the given assumptions, the whole closed interval determined by the extremal points of the ranges of $\kappa_1$ and $\kappa_2$ belong to the spectrum of $\bc{L}^{-1}\bc{A}$. Consequently, when the ranges of $\kappa_1$ and $\kappa_2$ are disjoint, the spectrum of $\bc{L}^{-1}\bc{A}$ contains also the interval between them. Please note that here it is not assumed that $K$ is continuous throughout the closure $\overline{\Omega}$ and that the subdomain $S$ is of an arbitrarily small size.

\subsubsection{Tensors continuous at least at a single point}\label{th:disjoint}
The following lemma refines further the assumptions under which the statement of \cref{th:constant} holds.
\begin{lemma} 
\label{th:interval-theorem}
Assume that the diagonal tensor \cref{eq:diag:tensor} with the bounded and Lebesgue integrable functions $\kappa_i$, $i=1,2$, is continuous (at least) at a single point in $\Omega$.
If
\begin{equation}\label{eq:int:ass}
\sup_{(x,y)\in\Omega} \kappa_1(x,y) < \inf_{(x,y)\in\Omega}\kappa_2(x,y),
\end{equation}
then the following closed interval belongs to the spectrum of $\bc{L}^{-1}\bc{A}$,
\begin{equation}\label{eq:int}
[\sup_{(x,y)\in\Omega} \kappa_1(x,y), \, \inf_{(x,y)\in\Omega} \kappa_2(x,y)]  \subset \mathrm{sp}(\mathcal{L}^{-1} \mathcal{A}).
\end{equation}
The analogous statement obviously holds with interchanging the roles of $\kappa_1$ and $\kappa_2$ in \cref{eq:int:ass} and \cref{eq:int}.  
\end{lemma}

\begin{proof}
We will prove the statement by contradiction. Consider
\[\lambda \in [\sup_{(x,y)\in\Omega} \kappa_1(x,y), \, \inf_{(x,y)\in\Omega} \kappa_2(x,y)] \]
such that $\lambda \notin \mathrm{sp}(\bc{L}^{-1} \bc{A})$, i.e., such that the operator $\bc{L}^{-1} \bc{A} - \lambda \bc{I}$ has a bounded inverse.

Let $(x_0,y_0)\in\Omega$ be the point of continuity of the tensor $K(x,y)$. Applying \cref{th:constant} to the preconditioned operator $\bc{L}^{-1}\bc{A}_l$, where $\bc{A}_l$ is defined for any sufficiently small $l$ by
\[\langle \mathcal{A}_l \phi, \psi \rangle \equiv \int_{\Omega} K_l \nabla \phi  \cdot  \nabla \psi, \quad \phi, \psi \in H_0^1(\Omega)\]
and $K_l(x,y)$ is a local modification of $K$,  
\begin{align*}
K_l(x,y)&\equiv \left\{ 
\begin{array}{ll}
K(x_0,y_0), & (x,y) \in S_l, \\
K(x,y), & (x,y) \in \Omega \setminus S_l,
\end{array}
\right. \\
S_l &= (x_0,x_0+l) \times (y_0,y_0+l), 
\end{align*}
yields that
\begin{equation}\label{eq:proof:spectrum}
\lambda  \in  \mathrm{sp}(\bc{L}^{-1} \bc{A}_l).
\end{equation}
On the other hand, since we assume that $\bc{L}^{-1}\bc{A} - \lambda\bc{I}$ is invertible, 
\begin{equation}
\begin{aligned}
\bc{L}^{-1}\bc{A}_l - \lambda\bc{I} 
	&= (\bc{L}^{-1}\bc{A} - \lambda\bc{I}) + (\bc{L}^{-1}\bc{A}_l - \bc{L}^{-1}\bc{A}) \\
	&= (\bc{L}^{-1}\bc{A} - \lambda\bc{I})[\bc{I} + (\bc{L}^{-1}\bc{A} - \lambda\bc{I})^{-1}\bc{L}^{-1}(\bc{A}_l - \bc{A})].
\end{aligned}\nonumber
\end{equation}
In \cref{BFnew2} we prove that for sufficiently small $r>0$
\begin{equation}\label{eq:neumann:con}
\|(\bc{L}^{-1}\bc{A} - \lambda\bc{I})^{-1}\bc{L}^{-1}(\bc{A}_l - \bc{A})\|_{\bc{L}} < 1,
\end{equation}
and the Neumann series argument therefore ensures that $\bc{L}^{-1}\bc{A}_l-\lambda\bc{I}$ has a bounded inverse. Consequently, $\lambda\notin \mathrm{sp}(\bc{L}^{-1} \bc{A}_l)$, which contradicts \cref{eq:proof:spectrum}. (Inequality \cref{eq:neumann:con} holds due to the assumption that $\lambda \notin \mathrm{sp}(\bc{L}^{-1} \bc{A})$ and due to the continuity of $K(x,y)$ at the point $(x_0,y_0)$. See \cref{BFnew2} for further details).
\end{proof}

It is worth noting that the statement of \cref{th:interval-theorem} requires continuity of the tensor $K$ only at an arbitrary {\em single point} belonging to $\Omega$.

\section{Continuous diagonal tensors}\label{sec:diag:ii}
We first complement \cref{th:function:values}, and Theorem~3.1 in \cite{Nie09}, by proving the `reverse inclusion'.
\subsection{The spectrum is a subset of the extremal interval}
\begin{lemma}\label{th:inclusion}
Assume that the diagonal tensor \cref{eq:diag:tensor} is continuous throughout the closure $\overline{\Omega}$. Then
\begin{equation}
 \mathrm{sp}(\mathcal{L}^{-1} \mathcal{A})) \subset \mathrm{Conv}(\kappa_1(\overline{\Omega}) \cup \kappa_2(\overline{\Omega})). 
\end{equation}
\end{lemma}
\begin{proof}
Using the self-adjointness \cref{eq:self:L} of the operator $\bc{L}^{-1}\bc{A}$, we can use the standard results from the theory of self-adjoint operators (see, e.g., \cite[Section 6.5]{1982Friedman_Book}) and conclude that the spectrum of $\bc{L}^{-1}\bc{A}$ is real and that 
\begin{align}
\nonumber
 \mathrm{sp}(\mathcal{L}^{-1} \mathcal{A})
&\subset
\left[\inf_{u\in H_0^1(\Omega)} \frac{(\bc{L}^{-1}\bc{A}u,u)_{\bc{L}}}{(u,u)_{\bc{L}}},\
		\sup_{u\in H_0^1(\Omega)} \frac{(\bc{L}^{-1}\bc{A}u,u)_{\bc{L}}}{(u,u)_{\bc{L}}}\right] \\
		&=
	\left[\inf_{u\in H_0^1(\Omega)} \frac{\langle\bc{A}u,u\rangle}{\langle\bc{L}u,u\rangle},\
		\sup_{u\in H_0^1(\Omega)} \frac{\langle\bc{A}u,u\rangle}{\langle\bc{L}u,u\rangle}\right].
\end{align}
Moreover, the endpoints of this interval are contained in the spectrum.

It remains to bound
\begin{equation}\label{eq:values}
\frac{\langle\bc{A}u,u\rangle}{\langle\bc{L}u,u\rangle}
\end{equation} in terms of the extreme values of the scalar functions $\kappa_1$ and $\kappa_2$. Since $ u_x^2(x,y)\geq 0$ and $ u_y^2(x,y)\geq 0$, we can bound   \cref{eq:values} as follows 
\begin{align}
\nonumber
\sup_{u\in H_0^1(\Omega)} \frac{\langle\bc{A}u,u\rangle}{\langle\bc{L}u,u\rangle}
	&= \sup_{u\in H_0^1(\Omega)} \frac{\int_{\Omega} K \nabla u \cdot \nabla u}{\int_{\Omega} \|\nabla u\|^2}
	= \sup_{u\in H_0^1(\Omega)} \frac{\int_{\Omega} \kappa_1 u_x^2 + \kappa_2 u_y^2}{\int_{\Omega} \|\nabla u\|^2}\\
\nonumber
	&\leq \sup_{u\in H_0^1(\Omega)} \frac{\int_{\Omega} \sup_{(x,y) \in \Omega}\max_{i=1,2}\{ \kappa_i(x,y) \} \, \| \nabla u \|^2}{\int_{\Omega} \|\nabla u\|^2}  \\
	\label{D2}
	&\leq \sup_{(x,y) \in \Omega}  \max_{i=1,2} \{\kappa_i(x,y)\}.
\end{align}
Similarly,
\begin{align*}
\inf_{u\in H_0^1(\Omega)} \frac{\langle\bc{A}u,u\rangle}{\langle\bc{L}u,u\rangle}	&\geq \inf_{(x,y) \in \Omega} \min_{i=1,2} \{\kappa_i(x,y)\}.
\end{align*}
For $K(x,y)$ continuous on $\overline{\Omega}$, the infimum and supremum of its components $\kappa_1(x,y)$ and $\kappa_2(x,y)$ are attained. Please notice that there is no assumption about the positive (negative) definiteness of $K$.
\end{proof}
\noindent
We are now ready to prove \cref{th:theorem} for continuous diagonal tensors.

\subsection{Main result --  diagonal tensors}\label{sec:maindiag}
\begin{theorem}
\label{th:theorem:diag}
Consider an open and bounded Lipschitz domain $\Omega\subset \nmbr{R}^2$.
If the diagonal tensor \cref{eq:diag:tensor} is continuous throughout the closure $\overline{\Omega}$, 
then 
\begin{equation}
 \mathrm{sp}(\mathcal{L}^{-1} \mathcal{A})) = \mathrm{Conv}(\kappa_1(\overline{\Omega}) \cup \kappa_2(\overline{\Omega})).\nonumber
\end{equation}
\end{theorem}
\begin{proof}
Assume  that the diagonal tensor $K(x,y)$ is continuous throughout $\overline{\Omega}$. Then, by \cref{th:function:values,th:constant}, 
\begin{align}\nonumber
\mathrm{Conv}(\kappa_1(\Omega) \cup \kappa_2(\Omega)) \subset\  &\mathrm{sp}(\mathcal{L}^{-1} \mathcal{A}), 
\end{align}
and due to the continuity of $K(x,y)$, and the fact that $\mathrm{sp}(\mathcal{L}^{-1} \mathcal{A})$ is a closed set (see, e.g., \cite{BRen93}),
\begin{align}\nonumber
\mathrm{Conv}(\kappa_1(\overline{\Omega}) \cup \kappa_2(\overline{\Omega})) \subset\  &\mathrm{sp}(\mathcal{L}^{-1} \mathcal{A}).
\end{align}%
Finally, by \cref{th:inclusion},
\begin{align}\nonumber
&\mathrm{sp}(\mathcal{L}^{-1} \mathcal{A}) \subset \mathrm{Conv}(\kappa_1(\overline{\Omega}) \cup \kappa_2(\overline{\Omega})),
\end{align}
which gives the statement.
\end{proof}

\section{Proof of \cref{th:theorem}}\label{sec:symm}
It remains to revisit and complete the arguments given above for the general self-adjoint operator in \cref{eq:eigen}. 
Consider the general symmetric tensor 
\begin{equation}
\label{eq:symm:tensor}
K(x,y)=\left[ 
\begin{array}{cc}
k_1(x,y) & k_3(x,y) \\
k_3(x,y) & k_2(x,y)
\end{array}
\right],
\end{equation}
where  $k_1$, $k_2$ and $k_3$ are bounded and Lebesgue integrable functions defined on $\Omega$,
with the spectral decomposition  
\begin{equation}
\label{eq:symm:tensor:dec}
K(x,y)=Q(x,y) \left[ 
\begin{array}{cc}
\kappa_1(x,y) & 0 \\
0 & \kappa_2(x,y)
\end{array}
\right]  Q^T(x,y);
\end{equation}
see \cref{eq:decomposition}.

The structure of the proof of \cref{th:theorem} is  fully analogous to the proof of \cref{th:theorem:diag} formulated for diagonal tensors. We will now restate the associated lemmas for the general case and comment on the technical differences that must be considered. 

For convenience we will use, when appropriate, the  column vector notation
\begin{equation}\nonumber
\alg{w} = (x,y)^T,\quad (x,y)\in\Omega,
\end{equation}
and for any function $f$ defined on $\Omega$ its gradient $\nabla f$ will be considered as a column vector.

\begin{lemma}[see \cref{th:function:values}]\label{th:function:values:ii}
Consider the symmetric tensor \cref{eq:symm:tensor} with the spectral decomposition \cref{eq:symm:tensor:dec}. If the tensor $K$ is continuous at $(x_0,y_0) \in \Omega$, then 
\[
\kappa_i(x_0,y_0) \in \mathrm{sp}(\mathcal{L}^{-1} \mathcal{A}),\quad i=1,2. 
\]
\end{lemma}
\begin{proof}
We will use the following notation for the spectral decomposition of $K(x,y)$ at the point of continuity $(x_0,y_0)$:
\begin{align*}
& K_0 \equiv K(x_0,y_0) = Q_0 \Lambda_0 Q_0^T, \quad Q_0 \equiv Q(x_0,y_0),\quad Q_0^TQ_0=I,\\
& \Lambda_0 \equiv \Lambda(x_0,y_0) = \mbox{diag}(\kappa_1(x_0,y_0),\kappa_2(x_0,y_0)).
\end{align*}

Simple algebraic computations give that, for any $(x,y)\in\Omega$,
\begin{equation}\label{eq:roots}
\kappa_{1} = \tfrac{1}{2}(k_1 + k_2 + \sqrt{D}),\quad \kappa_{2} = \tfrac{1}{2}(k_1 + k_2 - \sqrt{D})
\end{equation}
where $D = (k_1-k_2)^2+4k_3^2$.
Therefore, at any point of continuity of the tensor $K(x,y)$, the functions $\kappa_1(x,y)$ and $\kappa_2(x,y)$ are also continuous.

For sufficiently small $r$, consider the closed neighborhood $R_r$ defined in \cref{eq:domain:r} and its counterpart defined as
\begin{equation}\nonumber
S_r =  \{ Q_0 \alg{z} \; | \; \alg{z} \in R_r \},
\end{equation}
where the choice of $r$ in \cref{eq:domain:r} ensures that both $R_r  \subset \Omega$ and $S_r  \subset \Omega$. Consider the functions
\begin{equation}\nonumber
\tilde{v}_r(\alg{w}) \equiv v_r(Q_0^T \alg{w}), \quad  \alg{w}\in \Omega, 
\end{equation}
where  $v_r$ is defined in \cref{B0}.
Since $|\mbox{det}\, Q|=1$, the change of variables gives
\begin{equation}
\|\tilde{v}_r\|_{\bc{L}}^2 = \int_{S_r} \|\nabla \tilde{v}_r(\alg{w})\|^2 d\alg{w} = \int_{R_r} \|\nabla v_r(\alg{z})\|^2 d\alg{z}=  \|v_r\|_{\bc{L}}^2,
\end{equation}
and, from \cref{B2},
\begin{equation}\label{eq:pr:lim:w}
\lim_{r\to 0} \|\tilde{v}_r\|_{\bc{L}}  = 2 \neq 0.
\end{equation}

Analogously to \cref{B3} we consider
\begin{equation}\nonumber
u_r \equiv (\lambda \mathcal{I} - \mathcal{L}^{-1} \mathcal{A}) \tilde{v}_r,\quad \lambda \equiv \kappa_1(x_0,y_0),  
\end{equation}
with the norm
\begin{align}
 \| u_r \|_{\bc{L}}^2 &= \int_{\Omega} (\lambda I - K) \nabla \tilde{v}_r  \cdot  \nabla u_r,  \\
&= \int_{S_r} (\lambda I - K_0) \nabla \tilde{v}_r  \cdot  \nabla u_r 
+ \int_{S_r} (K_0 - K) \nabla \tilde{v}_r  \cdot  \nabla u_r.\label{eq:two:integrals}
\end{align}
Our goal is to show that if $\lambda\notin \mathrm{sp}(\mathcal{L}^{-1} \mathcal{A})$, then $\lim_{r\to0}\|u_r\|_{\bc{L}} = 0$, which contradicts \cref{eq:pr:lim:w}.
Concerning the second integral in \cref{eq:two:integrals},
\begin{align*}
\int_{S_r} (K_0 - K) \nabla \tilde{v}_r  \cdot  \nabla u_r &
\leq \sup_{\alg{w} \in S_r } \|K_0 - K(\alg{w}) \|  \, \|\tilde{v}_r\|_{\bc{L}}\,  \|u_r\|_{\bc{L}}.
\end{align*}
Using the continuity of $K(x,y)$ at the point $(x_0,y_0)$ and the fact that $\|\tilde{v}_r\|_{\bc{L}}\|u_r\|_{\bc{L}}$ is bounded, the second integral on the right hand side of \cref{eq:two:integrals} vanishes as $r\to0$.
For the remaining term in \cref{eq:two:integrals} we find that 
\begin{align*}
\int_{S_r} (\lambda I - K_0) \nabla \tilde{v}_r  \cdot  \nabla u_r
&= \int_{S_r} Q_0(\lambda I  - \Lambda_0) Q_0^T \nabla \tilde{v}_r  \cdot  \nabla u_r\\
%
&\leq  \left( \int_{S_r} \|Q_0(\lambda I - \Lambda_0) Q_0^T \nabla \tilde{v}_r\|^2  \right)^{1/2} \, \| u_r \|_{\bc{L}}.
\end{align*}
Applying the chain rule  gives $\nabla \tilde{v}_r(\alg{w}) = Q_0 \nabla v_r({Q_0^T}\alg{w}) = Q_0 \nabla v_r(\alg{z})$, which together with orthogonality of $Q_0$ gives (considering $\lambda = \kappa_1(x_0,y_0)$)
\begin{align*}
\int_{S_r} \|Q_0(\lambda I - \Lambda_0) Q_0^T \nabla \tilde{v}_r\|^2  &
= \int_{S_r} \|(\lambda I - \Lambda_0) \nabla v_r(Q_0^T\alg{w})\|^2   \\
&= \int_{R_r} \|(\lambda I - \Lambda_0) \nabla v_r(\alg{z})\|^2  	\\
&= \int_{R_r} \|(\lambda  - \kappa_2(x_0,y_0)) (v_r)_y(\alg{z})\|^2 	\\
&\leq |\lambda - \kappa_2(x_0,y_0)| \, \|(v_r)_y\|^2_{L^2(\Omega)}, 
\end{align*}
where the upper bound vanishes as $r\to0$ due to \cref{B1}.

The proof that $\kappa_2(x_0,y_0)$ belongs to the spectrum of the preconditioned operator, provided that the assumptions of the lemma hold, is trivially analogous. 
\end{proof}

\noindent
The remaining part of the proof of Theorem 1.1 is a straightforward extension of the analysis presented in \cref{sec:diag:ii}.

\begin{lemma}[see \cref{th:constant}]
\label{th:constant:ii}
Consider a symmetric tensor \cref{eq:symm:tensor}, with bounded and Lebesgue integrable functions $k_i$, $i = 1,2,3$, which are constant on an open subdomain $S\subset\Omega$. Assuming that
\begin{equation}
\label{eq:constant:ii:ass}
\sup_{(x,y)\in\Omega} \kappa_1(x,y) < \inf_{(x,y)\in\Omega}\kappa_2(x,y),
\end{equation}
then the following closed interval belongs to the spectrum of $\bc{L}^{-1}\bc{A}$,
\begin{equation}\label{eq:constant:ii}
[\sup_{(x,y)\in\Omega} \kappa_1(x,y), \, \inf_{(x,y)\in\Omega} \kappa_2(x,y)]  \subset \mathrm{sp}(\mathcal{L}^{-1} \mathcal{A}).
\end{equation}
The analogous statement obviously holds with interchanging the roles of $\kappa_1$ and $\kappa_2$ in \cref{eq:constant:ii:ass} and \cref{eq:constant:ii}.
\end{lemma}

\begin{proof}
Since $K(x,y)$ and its spectral decomposition $K = \bar{Q} \bar{\Lambda} \bar{Q}^T$ are constant on $S$,
the change of variables $\alg{w}  = \bar{Q}\alg{z}$ 
transforms the eigenvalue problem \cref{eq:eigen} in the subdomain $S$ to the form
\[
\nabla_{\alg{z}} \cdot (\bar{\Lambda} \nabla_{\alg{z}} v) = \lambda \Delta_{\alg{z}} v \quad \mbox{in } 
R =  \{  \bar{Q}^T \alg{w} \, | \, \alg{w} \in S \},
\]
where the diagonal tensor $\bar{\Lambda}$ is constant. Employing the argument used to prove \cref{th:constant} finishes the proof.
\end{proof} 

\begin{lemma}[see \cref{th:interval-theorem}] 
\label{th:interval-theorem:ii}
Assume that the symmetric tensor \cref{eq:symm:tensor} with the bounded and Lebesgue integrable functions $k_i$, $i=1,2,3$, is continuous at least at a single point in $\Omega$.
Assuming that
\begin{equation}\label{eq:int:ii:ass}
\sup_{(x,y)\in\Omega} \kappa_1(x,y) < \inf_{(x,y)\in\Omega}\kappa_2(x,y),
\end{equation}
then the following closed interval belongs to the spectrum of $\bc{L}^{-1}\bc{A}$,
\begin{equation}\label{eq:int:ii}
[\sup_{(x,y)\in\Omega} \kappa_1(x,y), \, \inf_{(x,y)\in\Omega} \kappa_2(x,y)]  \subset \mathrm{sp}(\mathcal{L}^{-1} \mathcal{A}).
\end{equation}
The analogous statement obviously holds with interchanging the roles of $\kappa_1$ and $\kappa_2$ in \cref{eq:int:ii:ass} and \cref{eq:int:ii}.
\end{lemma}

\begin{proof}
The proof is fully analogous to the proof of \cref{th:interval-theorem}.
\end{proof}

\begin{lemma}[see \cref{th:inclusion}]\label{th:inclusion:ii}
Let the symmetric tensor \cref{eq:symm:tensor} be continuous throughout the closure $\overline{\Omega}$. Then
\begin{equation}\nonumber
 \mathrm{sp}(\mathcal{L}^{-1} \mathcal{A})) \subset \mathrm{Conv}(\kappa_1(\overline{\Omega}) \cup \kappa_2(\overline{\Omega})). 
\end{equation}
\end{lemma}
\begin{proof}
The proof is analogous to the proof of \cref{th:inclusion}
with the argument used in the derivation of \cref{D2} now written in the form
\[K\nabla u\cdot\nabla u = \Lambda Q^T \nabla u \cdot Q^T\nabla u 
\leq \sup_{\alg{w} \in \Omega}\max_{i=1,2} \{\kappa_i(\alg{w})\} \|Q^T\nabla u\|^2.\]
Due to the orthogonality of $Q$ we get
\begin{align*}
\int_{\Omega} K \nabla u \cdot \nabla u  &\leq \sup_{\alg{w} \in \Omega} \max_{i=1,2}\{\kappa_i(\alg{w})\}\int_{\Omega} \|\nabla u\|^2,
\end{align*}
and, similarly,
\[\inf_{\alg{w} \in \Omega}  \max_{i=1,2} \{\kappa_i(\alg{w})\}\int_{\Omega} \|\nabla u\|^2 \leq \int_{\Omega} K \nabla u \cdot \nabla u.\]
\end{proof}

\noindent
The proof of \cref{th:theorem} is completed by combination of  \cref{th:function:values:ii,th:constant:ii,th:interval-theorem:ii,th:inclusion:ii}; see the proof of \cref{th:theorem:diag}.


\section{Neumann boundary conditions}\label{sec:neumann}
\Cref{th:theorem} also holds for generalized eigenvalue problems with homogeneous Neumann boundary conditions:  
\begin{equation}\label{eq:neum:problem}
\begin{split}
\nabla \cdot (K \nabla u) &= \lambda \Delta u \quad \mbox{ for } (x,y) \in \Omega, \\
\nabla u \cdot \alg{n} &= 0 \ \, \quad\quad\mbox{ for } (x,y) \in\partial\Omega,
\end{split}
\end{equation}
where $\alg{n}$ denotes the outwards pointing unit normal vector of $\partial \Omega$.
Instead of $H_0^1(\Omega)$, we now employ the space 
\begin{equation}\nonumber
V = \left\{ v \in H^1(\Omega)| \; \int_{\Omega} v  =0 \right\},
\end{equation}
with the  operators $\bc{L}$ and $\bc{A}$ defined analogously as above (see \cref{eq:L} and \cref{eq:A}) 
\begin{align*}
& \langle \mathcal{L} \phi, \psi \rangle = \int_{\Omega} \nabla \phi  \cdot  \nabla \psi,  \quad \phi,  \psi \in V, \\
& \langle \mathcal{A} \phi, \psi \rangle = \int_{\Omega} K \nabla \phi  \cdot  \nabla \psi, \quad \phi,  \psi \in V, 
\end{align*}
where $\bc{L}$ has a bounded inverse operator; see, e.g., \cite[Example 7.2.2, page 117]{BMar86}. 
For the Neumann problem, the functions $v_r$, defined in \cref{B0}, and the  solutions $\phi$  of the wave equation, defined in \cref{B5.01}, must be modified to
\[v_r - \int_{\Omega} v_r  \quad\mbox{ and }\quad  
\phi - \int_{\Sigma_l} \phi,
\]
respectively. 
The rest will follow in an analogous way to the analysis presented in this paper.

\section{Numerical experiments}\label{sec:exp}
In this section our theoretical results will be illuminated by numerical experiments where the matrices are constructed using FEniCS and the eigenvalues are computed and visualized with Matlab.\footnote{FEniCS version 2017.2.0 \cite{2015Fenics} and MATLAB Version: 9.5.0 (R2018b).} If not specified otherwise, we consider the domain $\Omega \equiv (0,1) \times (0,1)$ and a uniform triangular mesh with piecewise linear discretization basis functions is used. 
\label{experiments}

The examples in \cref{introduction} concerns  diagonal positive definite tensors. We first complement this by performing experiments with nondiagonal indefinite tensors. We consider three test problems in the form \cref{eq:eigen} with zero Dirichlet boundary conditions and with the following entries in the symmetric tensor \cref{eq:symm:tensor}:
\begin{equation*}\label{eq:exp:full}
\begin{aligned}
(P4):\quad k_1(x,y)  &= 5,\  k_2(x,y) =-5,\ k_3(x,y) = 0,\\
(P5):\quad k_1(x,y)  &= 3 ,\ k_2(x,y) = -3,\ k_3(x,y) = 4,\\
(P6):\quad k_1(x,y)  &= 3e^{-3(|x-0.5|+|y-0.5|)} ,\ k_2(x,y) = -k_1,\ k_3(x,y) = 4\cos(\tfrac{\pi(x+y-1)}{2}).
\end{aligned}
\end{equation*}
\begin{figure}[ht]
\centering
\includegraphics[width = 0.48\linewidth]{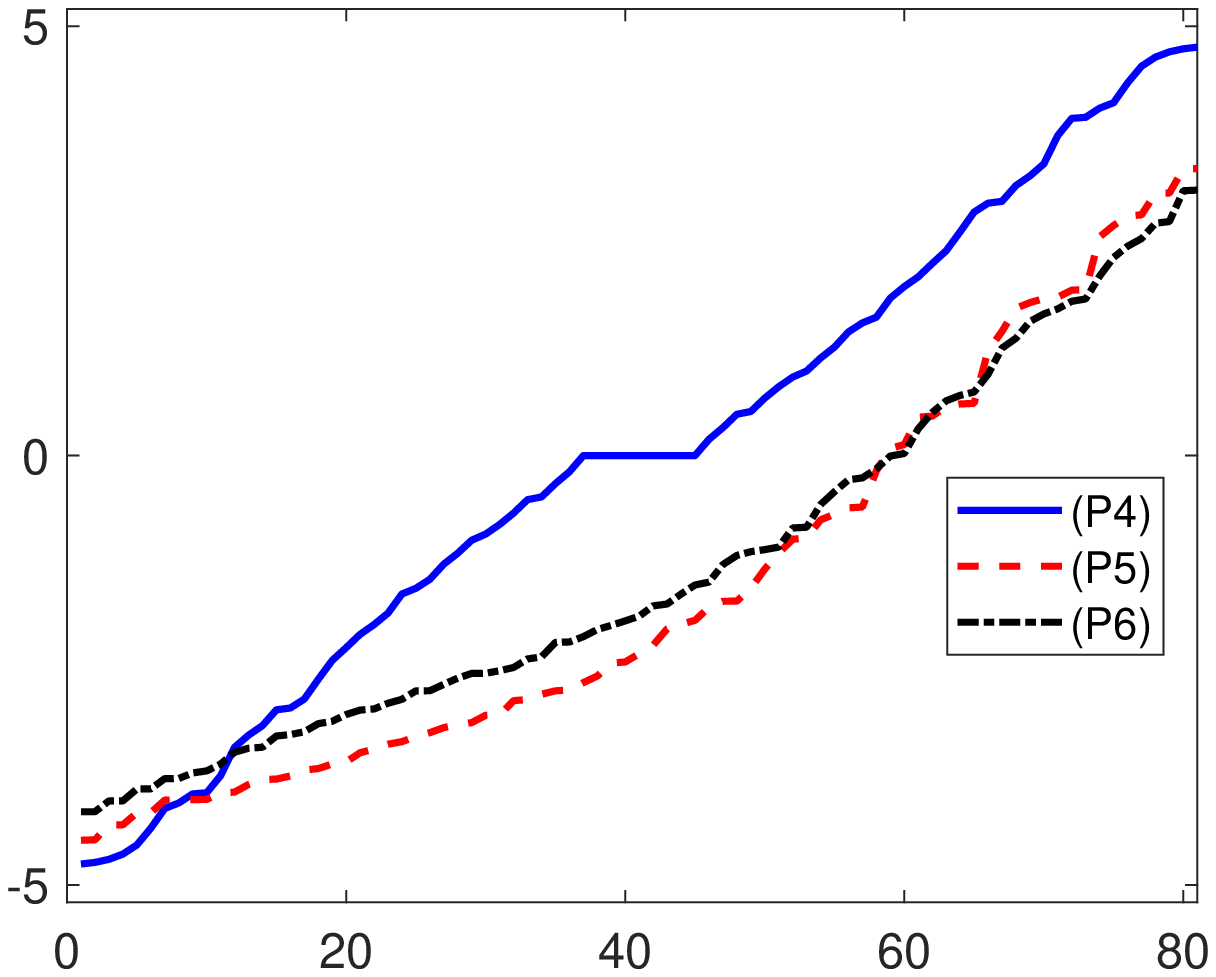}
\hspace{1em}
\includegraphics[width = 0.48\linewidth]{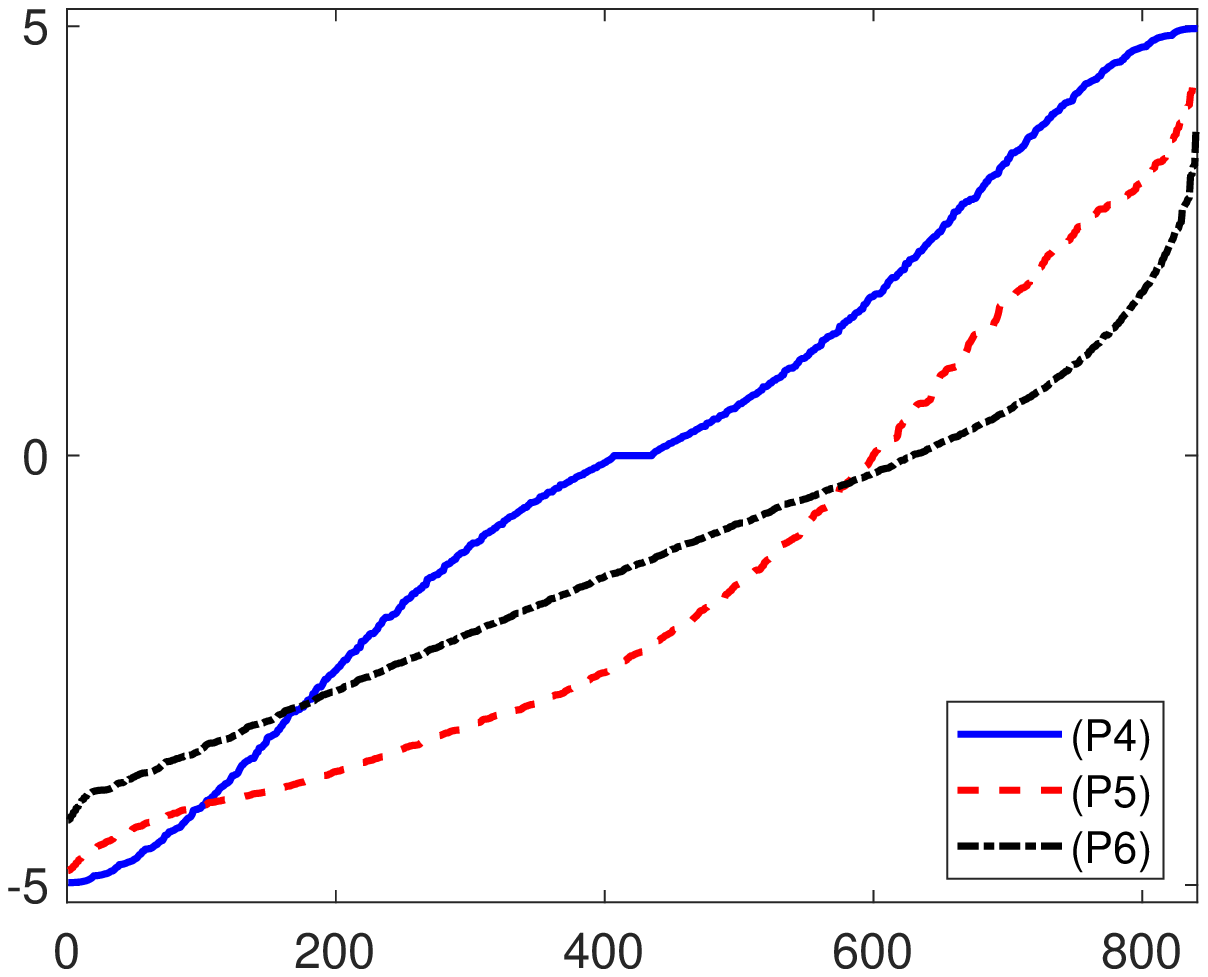}%
\caption{Spectra of the discretized test problems (P4), (P5) and (P6) for $N=81$ (left) and $N=841$ (right) degrees of freedom. Horizontal axis: the indices of the increasingly ordered eigenvalues. Vertical axis: the size of the eigenvalues.}
\label{fig:exp:full}
\end{figure}
Using  \cref{eq:roots} gives for the problems (P4) and (P5) that $\kappa_1(x,y) = -5$ and $\kappa_2(x,y) = 5$. Furthermore, for problem (P6) formula  \cref{eq:roots} yields  
\[\kappa_{1,2}(x,y)  = \pm\sqrt{k_1^2+k_3^2} = \pm\sqrt{9e^{-6(|x-0.5|+|y-0.5|)}+16\cos^2(\tfrac{\pi(x+y-1)}{2})},\] 
such that $\kappa_1(\overline{\Omega}) = - \kappa_2(\overline{\Omega}) = [3e^{-3},5]$. 
As in \cref{fig:intro}, the spectra visualized in \cref{fig:exp:full}  spread over the entire interval \cref{eq:conv:hull} defined by the nonoverlapping ranges $\kappa_1(\overline{\Omega})$ and $\kappa_2(\overline{\Omega})$. Refining the mesh gives better approximations of the endpoints of the interval $[-5,5]$. 
The fact that the tensor \cref{eq:symm:tensor} is not diagonal has no qualitative effect on the observed experimental data. We will therefore below only consider diagonal tensors.

The left part of \cref{fig:exp:nem} shows numerical results computed with homogeneous Neumann boundary conditions (see \cref{sec:neumann}). The results with zero Dirichlet boundary conditions are, for comparison, presented in the right part of \cref{fig:exp:nem}. We consider two test problems with the diagonal tensor \cref{eq:diag:tensor} defined by
\begin{equation}\label{eq:exp:nem}
\begin{aligned}
(P7):\quad \kappa_1(x,y)  &= 10-f(x,y),\quad \kappa_2(x,y) = 4+f(x,y),\\
(P8):\quad \kappa_1(x,y)  &= 8+f(x,y),\quad \kappa_2(x,y) = 6-f(x,y),
\end{aligned}
\end{equation}
where
\[f(x,y) = 4((x-0.5)^2+(y-0.5)^2)\]
is chosen such that, for both problems, $\kappa_1(\overline{\Omega}) = [8,10]$ and $\kappa_2(\overline{\Omega}) = [4,6]$. Note that these intervals do not overlap.
The minimum (respectively maximum) of the interval $[4,10]$ is obtained by the function $\kappa_1(x,y)$ (respectively $\kappa_2(x,y)$) in the interior of the solution domain for problem (P7), while for problem (P8) the endpoints of this interval are attained on the boundary $\partial \Omega$. In the latter case the endpoints of the interval $[4,10]$ are better approximated for the problem  with Neumamn boundary conditions. Similar behavior was also observed for other test cases.
\begin{figure}[ht]
\centering
\includegraphics[width = 0.49\linewidth]{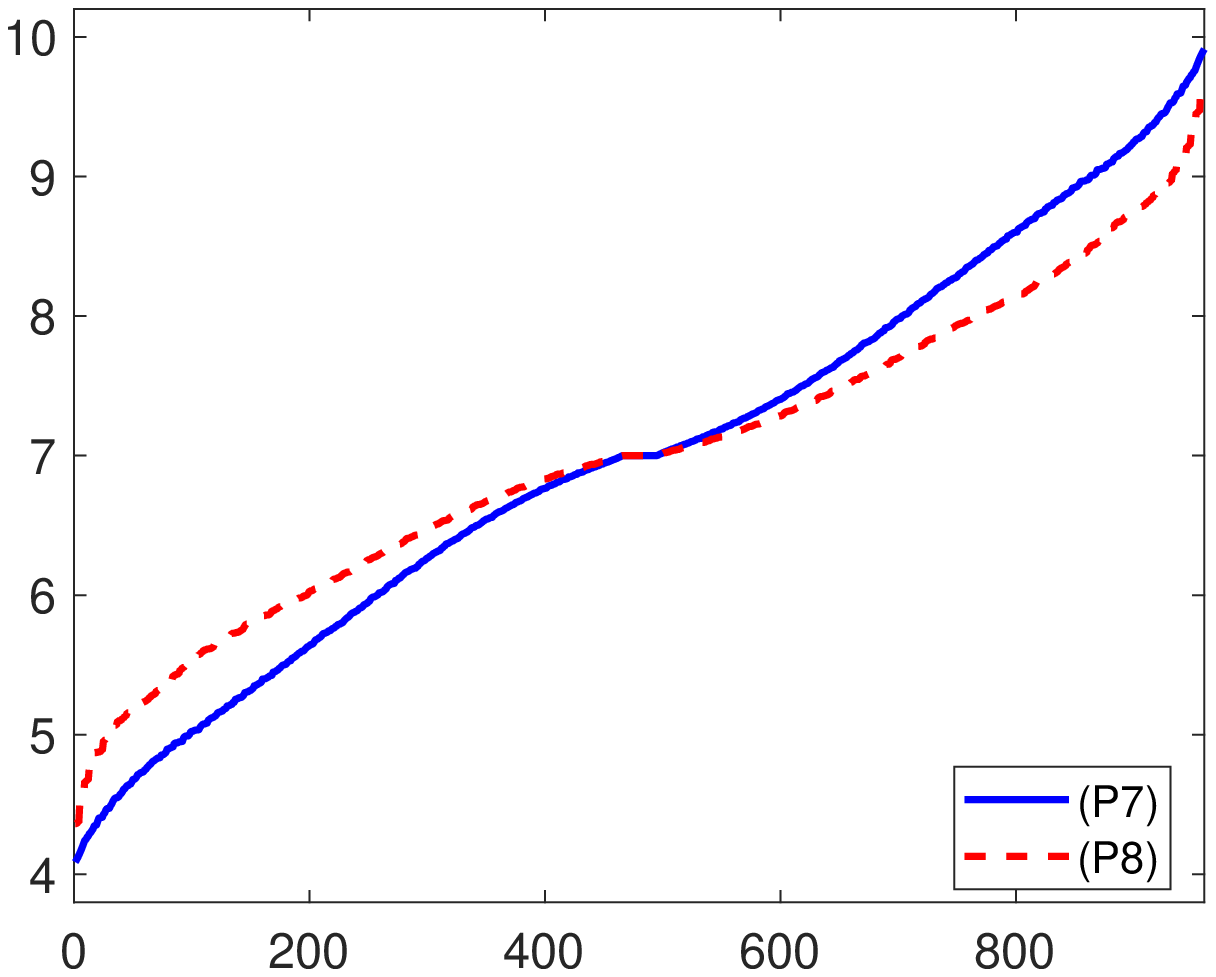}%
\hspace{1mm}
\includegraphics[width = 0.49\linewidth]{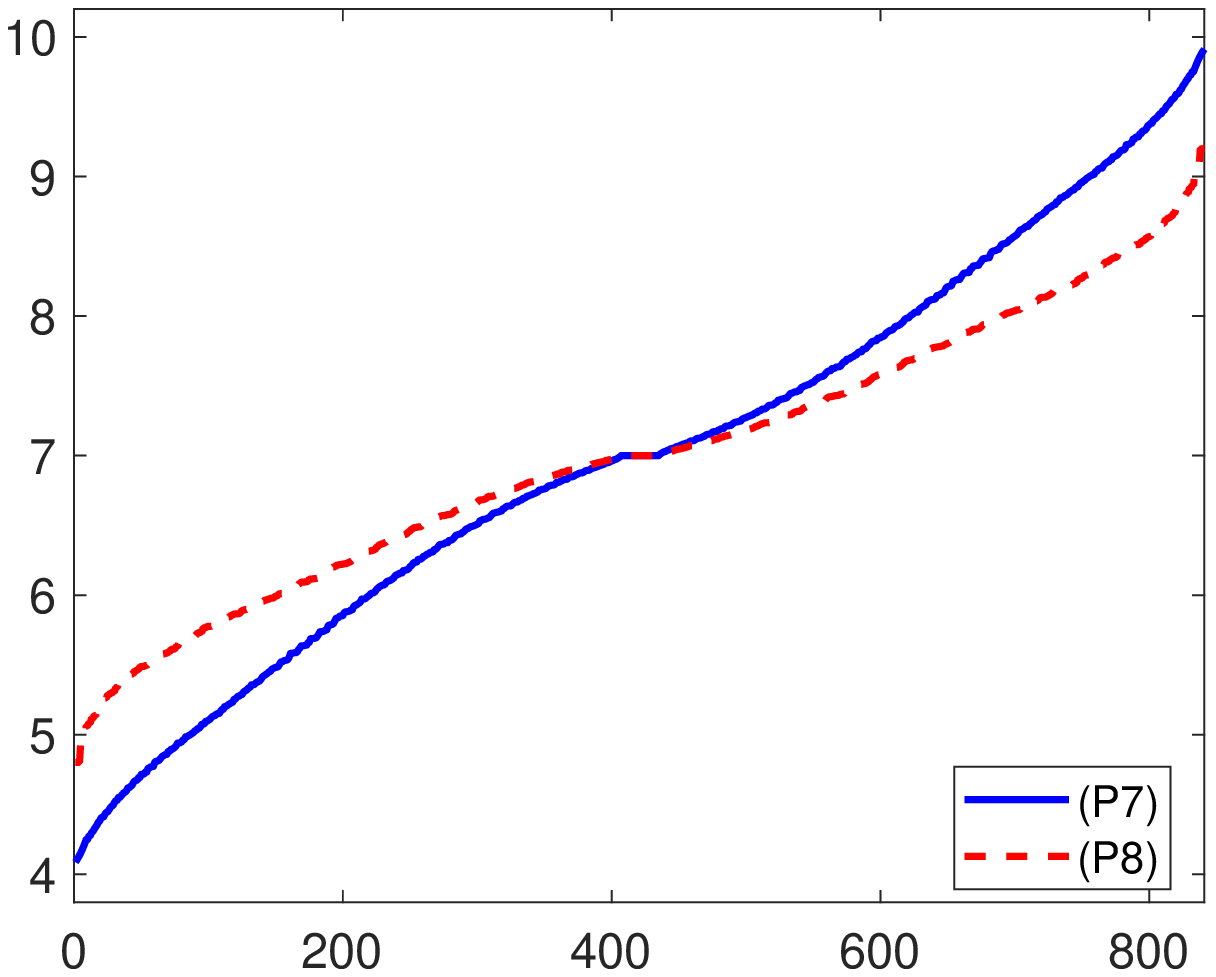}%
\caption{
Spectra of the discretized test problems (P7) and (P8) with zero Neumann boundary conditions (left) and zero Dirichlet  boundary conditions (right).}
\label{fig:exp:nem}
\end{figure}

Numerical results for nonconvex domains are presented in \cref{fig:exp:domain}.
We used the diagonal tensor  \cref{eq:diag:tensor} with
\begin{equation}\nonumber
\begin{aligned}(P9):\quad
\kappa_1(x,y)  &= 6-3e^{-3(|x-0.8|+|y-0.8|)},\quad \kappa_2(x,y) = 6+3e^{-3(|x-0.2|+|y-0.2|)}, 
\end{aligned}
\end{equation}
and the L-shaped domains
$\Omega_1  = (0,1)^2 \setminus (0,0.6)^2$ and $\Omega_2  =  (0,1)^2 \setminus (0.4,1)^2$;
see the illustration in the left part of \cref{fig:exp:domain}. We employed zero Dirichlet boundary conditions. The function $\kappa_1(x,y)$ (respectively $\kappa_2(x,y)$) has its minimum (respectively maximum) at the point $[0.8, 0.8]$ (respectively $[0.2, 0.2]$), which is outside the domain $\Omega_2$ (respectively $\Omega_1$).
As a result, we observe in \cref{fig:exp:domain} that the spectra of the disretized problems differ, depending on the ranges of functions $\kappa_1(x,y)$ and $\kappa_2(x,y)$ over   $\overline{\Omega}_1$ and $\overline{\Omega}_2$.
\begin{figure}[ht]
\centering
\begin{tikzpicture}[scale=1.3]
\path [draw=black] (0,6) -- (0,5.2) -- (1.2,5.2) -- (1.2,4) -- (2,4) -- (2,6) -- cycle;
\path [draw=black] (0,1.7) -- (2,1.7) -- (2,2.5) -- (0.8,2.5) -- (0.8,3.7) -- (0,3.7) -- cycle;
\node  at (1,5.6) {$\Omega_1$};
\node  at (1,2.1) {$\Omega_2$};
\node [color=white] at (1,1) {.};
\end{tikzpicture}
\hspace{1em}
\includegraphics[width = 0.7\linewidth]{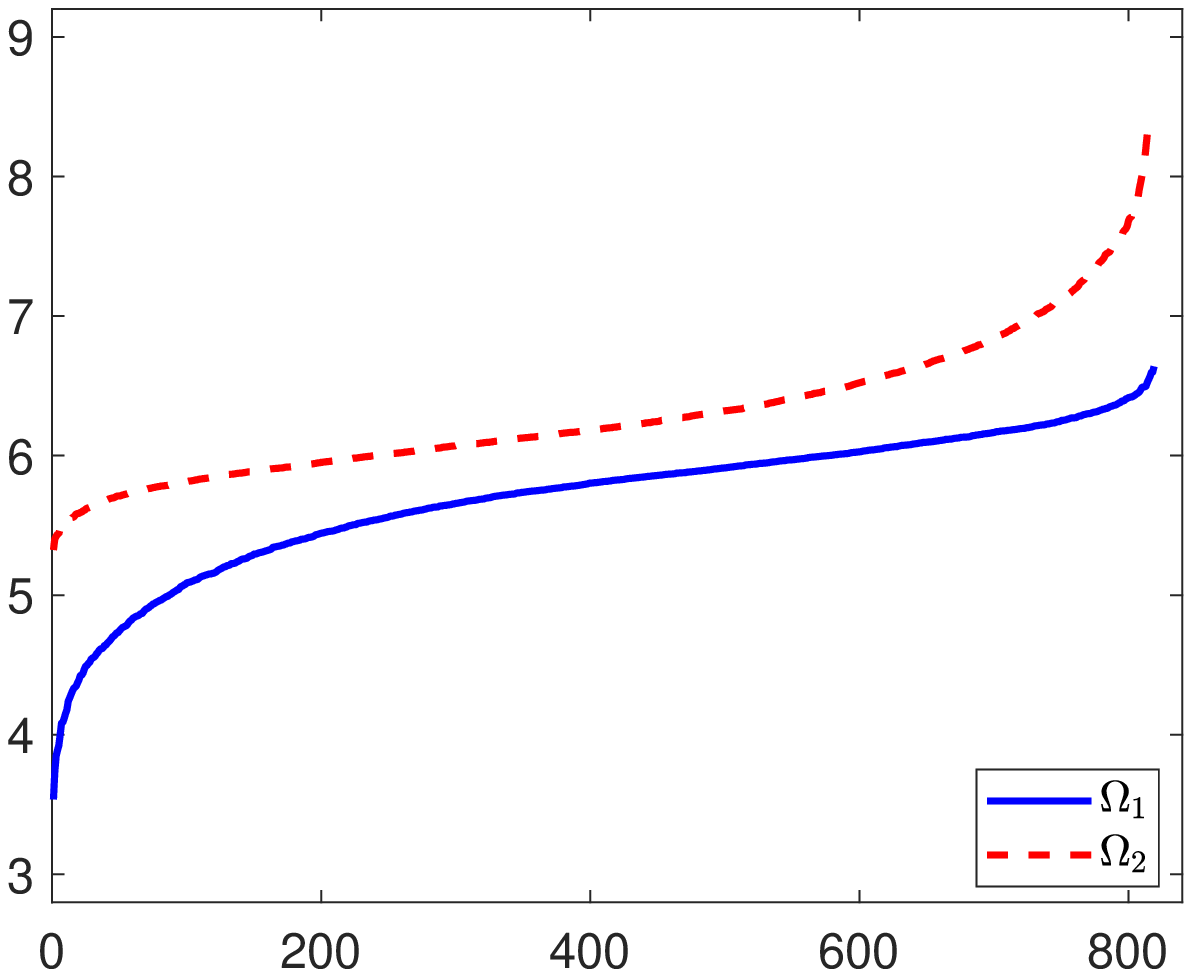}
\caption{Left: Illustration of the shapes of the domains $\Omega_1$ and $\Omega_2$. Right: Spectra of the test problem (P9) associated with the  domains $\Omega_1$ and $\Omega_2$.
The ranges satisfy $\kappa_1(\overline{\Omega}_1) \subset [3,6]$ and $\kappa_2(\overline{\Omega}_1) \subset [6,7]$ for the domain $\Omega_1$ and $\kappa_1(\overline{\Omega}_2) \subset [5,6]$   and $\kappa_2(\overline{\Omega}_2) \subset [6,9]$ for the domain $\Omega_2$.}
\label{fig:exp:domain}
\end{figure}

Finally, we present in \cref{fig:exp:3d} numerical results for 3D problems, which is not (yet) supported by rigorous proofs.
We consider the unit cube $\Omega \equiv (0,1)^3$, zero Dirichlet boundary conditions and the  diagonal tensor 
$K(x,y,z)  =  \mbox{diag}(\kappa_1, \kappa_2, \kappa_3)$ defined  as
\begin{equation*}\label{eq:exp:3d}
\begin{aligned}
(P10):&\ \kappa_1 = 1,\quad\kappa_2 = 5.5,\quad \kappa_3 = 10,\\
(P11):&\ 
\kappa_1 = 1+\sin^2(x+y+z), \
\kappa_2 = 5.5+\cos(\pi xyz), \
\kappa_3 = 10-\cos^2(x+y+z),
\\
(P12):&\ 
\kappa_1 = 1+(x+y+z-1)^2,\
\kappa_2 = 4+xy+z,\
\kappa_3 = 10 - 2(x+y+z-1)^2.
\end{aligned}
\end{equation*}
This choice of test problems follows the same `pattern' as for the introductory experiments presented in \cref{introduction}: 
The ranges of the functions $\kappa_i(x,y,z)$, $i =1,2,3$, are for (P10)  isolated points, they form nonoverlapping intervals for  (P11) and overlapping intervals for (P12). As for the 2D test cases, we observe that the spectra of the discretized problems are spread over the entire interval $[1,10]$, irrespectively of whether the associated ranges overlap or not.
\begin{figure}[ht]
\centering
\includegraphics[width = 0.7\linewidth]{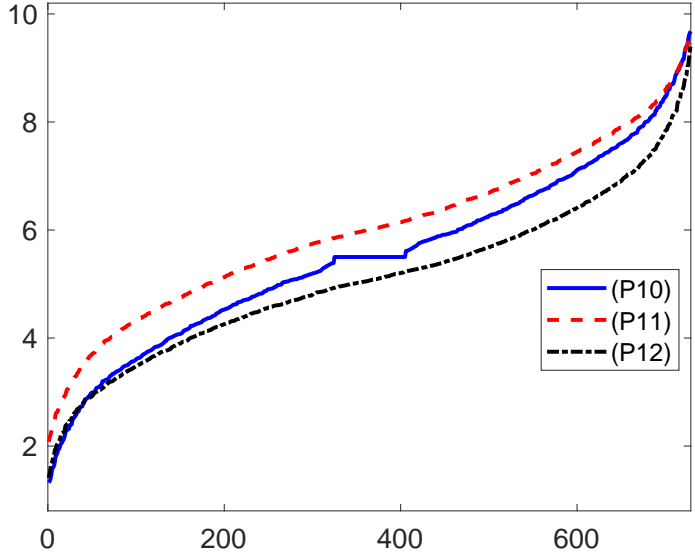}%
\caption{The spectra of the 3D test problems (P10)--(P12) spread over the entire interval $[1,10]$, while the ranges of the function entries of the diagonal tensors are as follows: Isolated points $\kappa_1(\overline{\Omega}) = 1,\
		   \kappa_2(\overline{\Omega}) = 5.5$ and
		 $\kappa_3(\overline{\Omega}) = 10$ for (P10),
nonoverlapping intervals $\kappa_1(\overline{\Omega}) = [1,2],\
			\kappa_2(\overline{\Omega}) = [4.5,6.5]$ and
		  $\kappa_3(\overline{\Omega}) = [9,10]$ for (P11) and
overlapping intervals $\kappa_1(\overline{\Omega}) = [1,5],\
			\kappa_2(\overline{\Omega}) = [4,6]$ and
		  $\kappa_3(\overline{\Omega}) = [2,10]$ for (P12).}
\label{fig:exp:3d}
\end{figure}

\section{Open problems} 
\label{sec:remarks}
In this paper we have rigorously analyzed $2$D problems, and it is an open question whether our main result \cref{th:theorem} also holds in $3$D, or even higher dimensions. Our numerical results indicate that such a generalization is possible, but, e.g., the task of construction functions similar to the $\{ v_r \}$ functions \cref{B0} will become more involved.

Another important issue is to `translate' our findings to discretized operators. This was accomplished in \cite{Ger19} for uniformly elliptic operators with scalar coefficient functions. That is,  \cite{Ger19} contains discrete versions of the results published 
in \cite{Nie09} and further develops towards approximating locally the individual eigenvalues. The techniques employed in \cite{Ger19} can be generalized to analyze discretized second order differential operators with indefinite tensors. Such a development is, however, out of the scope of this paper. An interesting question concerns the distribution of the eigenvalues: 
For discretized operators, are the eigenvalues evenly distributed in the interval \eqref{eq:conv:hull}? Our numerical experiments suggest that the answer may be positive. We will return to this question elsewhere.

\appendix
\section{Technical details about the inequalities \cref{B1} in the proof of \cref{th:function:values}}
\label{sec:apA}
We want to prove that, for sufficiently small $r>0$, 
\begin{align}\label{A:B1}
4-4r\leq \| (v_r)_x \|_{L^2(\Omega)}^2 &\leq 4, \\
\| (v_r)_y \|_{L^2(\Omega)}^2 &\leq 4r,
\label{A:B2}
\end{align}
where $v_r(x,y)$ is defined on $R_r$ by \cref{eq:domain:r,B0}.
Without  loss of generality, we consider the case $(x_0,y_0)=(0,0)$. Then $R_r = [-r^2,r^2]\times[-r,r]$ and
\begin{equation}\nonumber
v_r(x,y) = \sqrt{r}\min\left\{1-\tfrac{|x|}{r^2},\tfrac{1}{r}-\tfrac{|y|}{r^2}\right\} \quad
	\mbox{for }(x,y) \in R_r,
\end{equation}
\begin{figure}[ht]
\centering
\includegraphics[width=0.8\linewidth]{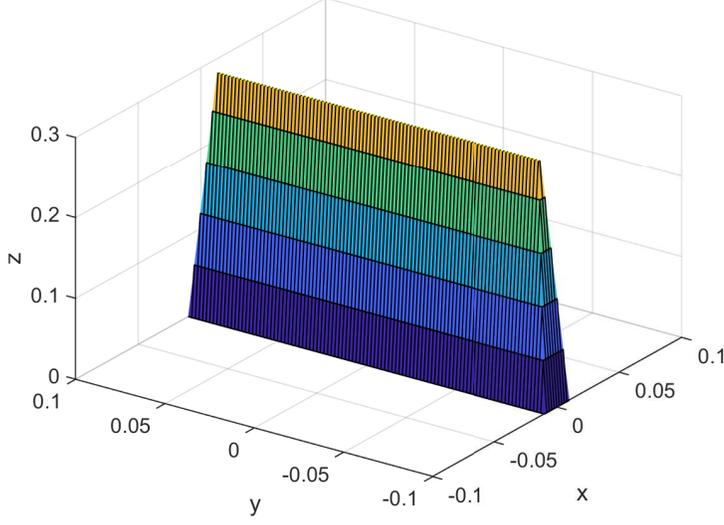}
\caption{\label{fig3} The function $v_r$ centered at the point $(0,0)$ with $r=0.1$.}
\end{figure}%
with $v_r(x,y)=0$ elsewhere; see \cref{fig3}.

For any $0<r<1$, the partial derivatives of  $v_r(x,y)$ are not defined at the boundary $\partial R_r$ of $R_r$, at the  set of points $\{(x,y) \in R_r: |y|-|x|=r-r^2\}$, and at the set of points $\{(x,y) :  x=0,\, |y|<r-r^2\}$ where $v_r(x,y)$ reaches its maximum; see the edges of $\{v_r(R_r)\}$ in \cref{fig3}. Simple computations yield that  within $R_r$
\begin{align*}
\left|\partial_x v_r(x,y)\right|^2 &= 0, \quad
\left|\partial_y v_r(x,y)\right|^2 = \tfrac{1}{r^3}, \quad
\mbox{for}\quad  |y|-|x|>r-r^2, \ (x,y) \notin \partial R_r,		\\
\left|\partial_x v_r(x,y)\right|^2 &= \tfrac{1}{r^3}, \quad
\left|\partial_y v_r(x,y)\right|^2 = 0, \quad
\mbox{for}\quad  x\neq0, \, |y|-|x|<r-r^2,\ (x,y) \notin \partial R_r.
\end{align*}
The upper bound in \cref{A:B1} thus holds because 
\begin{equation}
\label{eq:A2}
\|(v_r)_x \|_{L^2(\Omega)}^2 = \int_{R_r}|\partial_x v_r(x,y)|^2 
							   \leq\int_{R_r}\tfrac{1}{r^3}
							   = \tfrac{2r^2\cdot 2r}{r^3} = 4.
\end{equation}
Moreover, denoting 
\[
P_r = \{(x,y): x\neq 0,\, |x|<r^2,\, |y|<r-r^2\},
\]
we have
\[
\left|\partial_x v_r(x,y)\right|^2 = \tfrac{1}{r^3}, \quad
\left|\partial_y v_r(x,y)\right|^2 = 0, \quad
\mbox{for}\quad  (x,y)\in P_r.
\]
Thus $ \| (v_r)_x \|_{L^2(\Omega)}^2$ and $ \| (v_r)_y \|_{L^2(\Omega)}^2$ can be bounded as follows 
\begin{align*}
\| (v_r)_x \|_{L^2(\Omega)}^2 &=\int_{R_r}|\partial_x v_r(x,y)|^2 
								\geq \int_{P_r}|\partial_x v_r(x,y)|^2 
							   =\int_{P_r}\tfrac{1}{r^3}= \tfrac{2r^2\cdot 2(r-r^2)}{r^3} = 4-4r, \\
\| (v_r)_y \|_{L^2(\Omega)}^2 &=\int_{R_r}|\partial_y v_r(x,y)|^2 
							= \int_{R_r\setminus P_r}|\partial_y v_r(x,y)|^2 
							   \leq\int_{R_r\setminus P_r}\tfrac{1}{r^3} = \tfrac{2r^2\cdot 2r^2}{r^3} = 4r							   
\end{align*}
which completes the proof.

\section{Technical details about the bound \cref{eq:neumann:con} in the proof of \cref{th:interval-theorem}}%
\label{BFnew2}
Assume that $\bc{L}^{-1}\bc{A} - \lambda\bc{I}$ has a bounded inverse. We will show that, for sufficiently small $l>0$,
\begin{equation}\label{eq:app:2}
\|(\bc{L}^{-1}\bc{A} - \lambda\bc{I})^{-1}\bc{L}^{-1}(\bc{A}_l - \bc{A})\|_{\bc{L}} 
\leq \|(\bc{L}^{-1}\bc{A} - \lambda\bc{I})^{-1}\|_{\bc{L}}\|\bc{L}^{-1}(\bc{A}_l - \bc{A})\|_{\bc{L}}
< 1.
\end{equation}
The operator norm
\begin{equation}
\|\bc{L}^{-1}(\bc{A}_l - \bc{A})\|_{\bc{L}} 
	 \equiv \sup_{u\in H^1_0(\Omega)}
\frac{\|\bc{L}^{-1}(\bc{A}_l - \bc{A})u\|_{\bc{L}}} 
	 {\|u\|_{\bc{L}}} \nonumber
\end{equation}
can be expressed as (see, e.g. \cite[Theorem 4.1--3]{2013Ciarlet_Book})	
\begin{equation}
\|\bc{L}^{-1}(\bc{A}_l - \bc{A})\|_{\bc{L}} 
	 = \sup_{u,v\in H^1_0(\Omega)}
\frac{\left|\left(\bc{L}^{-1}(\bc{A}_l - \bc{A})u,v\right)_{\bc{L}}\right|}
	 {\|u\|_{\bc{L}}\|v\|_{\bc{L}}}. \label{eq:A:norm}
\end{equation}
Using
\begin{align*}
|\left(\bc{L}^{-1}(\bc{A}_l - \bc{A})u,v\right)_{\bc{L}}|
	&= 		\left|\left\langle(\bc{A}_l - \bc{A})u,v\right\rangle\right|					\\
	&= 		|\int_{S_l} (K(x_0,y_0)-K(x,y)) \nabla u \cdot  \nabla v|				\\
	&\leq 	\int_{S_l} \|K(x_0,y_0)-K(x,y)\| |\nabla u| \cdot |\nabla v|	\\
	&\leq 	\sup_{(x,y) \in S_l}\|K(x_0,y_0)-K(x,y)\|\, \|u\|_{\bc{L}}\|v\|_{\bc{L}},
\end{align*}
we get the bound
\begin{equation}
\nonumber
\|\bc{L}^{-1}(\bc{A}_l - \bc{A})\|_{\bc{L}} 
	 \leq \sup_{(x,y) \in S_l}\|K(x_0,y_0)-K(x,y)\|.
\end{equation}
Since $\|(\bc{L}^{-1}\bc{A} - \lambda\bc{I})^{-1}\|_{\bc{L}}$ is bounded, the continuity of $K(x,y)$ at the point $(x_0,y_0)$ ensures that $l$ can be chosen  such that \cref{eq:app:2} holds.

\bibliographystyle{siamplain}

\begin{thebibliography}{10}

\bibitem{2015Fenics}
{\sc M.~S. Aln{\ae}s, J.~Blechta, J.~Hake, A.~Johansson, B.~Kehlet, A.~Logg,
  C.~Richardson, J.~Ring, M.~E. Rognes, and G.~N. Wells}, {\em The fenics
  project version 1.5}, Archive of Numerical Software, 3 (2015),
  \url{https://doi.org/10.11588/ans.2015.100.20553}.

\bibitem{2013Ciarlet_Book}
{\sc P.~G. Ciarlet}, {\em Linear and nonlinear functional analysis with
  applications}, Society for Industrial and Applied Mathematics, Philadelphia,
  PA, 2013.

\bibitem{1982Friedman_Book}
{\sc A.~F. Friedman}, {\em Foundations of Modern Analysis}, Dover Publications
  Inc., 1982.

\bibitem{Ger19}
{\sc T.~Gergelits, K.~A. Mardal, B.~F. Nielsen, and Z.~Strako{\v{s}}}, {\em
  Laplacian preconditioning of elliptic {PDE}s: Localization of the eigenvalues
  of the discretized operator}, SIAM Journal on Numerical Analysis, 57 (2019),
  pp.~1369--1394.

\bibitem{Gre89}
{\sc A.~Greenbaum}, {\em Behavior of slightly perturbed {L}anczos and
  conjugate-gradient recurrences}, Linear Algebra Appl., 113 (1989), pp.~7--63,
  \url{https://doi.org/10.1016/0024-3795(89)90285-1}.

\bibitem{Hof17}
{\sc C.~Hofreither and S.~Takacs}, {\em Robust multigrid for isogeometric
  analysis based on stable splittings of spline spaces}, SIAM Journal on
  Numerical Analysis, 55 (2017), pp.~2004--2024.

\bibitem{Hof15}
{\sc C.~Hofreither, S.~Takacs, and W.~Zulehner}, {\em A robust multigrid method
  for isogeometric analysis in two dimensions using boundary correction},
  Computer Methods in Applied Mechanics and Engineering, 316 (2017),
  pp.~22--42.

\bibitem{LSB13}
{\sc J.~Liesen and Z.~Strako{\v{s}}}, {\em Krylov subspace methods: principles
  and analysis}, Numerical Mathematics and Scientific Computation, Oxford
  University Press, Oxford, 2012.

\bibitem{BMar86}
{\sc J.~T. Marti}, {\em Introduction to Sobolev Spaces and finite element
  solution of elliptic boundary value problems}, Academic Press, 1986.

\bibitem{Nie09}
{\sc B.~F. Nielsen, A.~Tveito, and W.~Hackbusch}, {\em Preconditioning by
  inverting the {L}aplacian; an analysis of the eigenvalues}, IMA Journal of
  Numerical Analysis, 29 (2009), pp.~24--42.

\bibitem{BRen93}
{\sc M.~Renardy and R.~C. Rogers}, {\em An Introduction to Partial Differential
  Equations}, Springer-Verlag, 1993.

\bibitem{San16}
{\sc G.~Sangalli and M.~Tani}, {\em Isogeometric preconditioners based on fast
  solvers for the {S}ylvester equation}, SIAM Journal on Scientific Computing,
  38 (2016), pp.~A3644--A3671.

\end{thebibliography}

\end{document}